\documentclass{smfart}

\usepackage[french]{babel}
\usepackage{amsmath,amssymb,amsfonts,amsthm}
\usepackage[all]{xy}
\usepackage{bbm,dsfont}
\usepackage{mathrsfs}

\title[Automorphismes naturels]{Automorphismes naturels de l'espace de Douady de points sur une surface}

\author{Samuel Boissi\`ere}

\date{\today}

\address{Laboratoire J.A.Dieudonn\'e UMR CNRS 6621,
         Universit\'e de Nice Sophia-Antipolis, Parc Valrose, 06108 Nice}

\email{sb@math.unice.fr}

\urladdr{http://math.unice.fr/$\sim$sb/}



\newtheorem{example}{Exemple}
\newtheorem{definition}{D\'efinition}
\newtheorem{lemma}{Lemme}
\newtheorem{proposition}{Proposition}
\newtheorem{corollaire}{Corollaire}
\newtheorem{remark}{Remarque}
\newtheorem{conjecture}{Conjecture}

\DeclareMathOperator{\Aut}{Aut}
\DeclareMathOperator{\Der}{Der}
\DeclareMathOperator{\End}{End}
\DeclareMathOperator{\Hom}{Hom}
\DeclareMathOperator{\Pic}{Pic}

\DeclareMathOperator{\id}{id}
\DeclareMathOperator{\ent}{e}

\DeclareMathOperator{\HilbS}{S^{[\bullet]}}

\DeclareMathOperator{\trace}{trace}
\DeclareMathOperator{\rg}{rg}

\DeclareMathOperator{\res}{res}

\newcommand{\vac}{|0\rangle}

\newcommand{\cD}{\mathcal{D}}
\newcommand{\cI}{\mathcal{I}}
\newcommand{\cL}{\mathcal{L}}
\newcommand{\cO}{\mathcal{O}}

\newcommand{\cZ}{\mathcal{Z}}

\newcommand{\IC}{\mathds{C}}
\newcommand{\IF}{\mathbb{F}}
\newcommand{\IH}{\mathbb{H}}
\newcommand{\IN}{\mathds{N}}
\newcommand{\IP}{\mathbb{P}}

\newcommand{\IZ}{\mathds{Z}}

\newcommand{\kf}{\mathfrak{f}}
\newcommand{\kh}{\mathfrak{h}}
\newcommand{\kg}{\mathfrak{g}}
\newcommand{\kq}{\mathfrak{q}}

\newcommand{\kS}{\mathfrak{S}}

\DeclareFontFamily{OT1}{pzc}{}
\DeclareFontShape{OT1}{pzc}{m}{n}{<-> s * [1.25] pzcmi7t}{}
\DeclareMathAlphabet{\mathpzc}{OT1}{pzc}{m}{n}
\newcommand{\pzm}{\mathpzc{m}}

\newcommand{\ie}{{\it i.e. }}

\begin{document}

\begin{abstract}
On \'etablit quelques r\'esultats g\'en\'eraux relatifs \`a la taille du groupe d'automorphismes de l'espace de Douady de points sur une surface, puis on \'etudie quelques propri\'et\'es des automorphismes provenant d'un automorphisme de la surface, en particulier leur action sur la cohomologie et la classification de leurs points fixes.
\end{abstract}

\maketitle

\section*{Introduction}

Autant l'on sait dire beaucoup sur les automorphismes des surfaces K3, notamment gr\^ace au th\'eor\`eme de Torelli global, autant en dimension sup\'erieure le groupe d'automorphismes des vari\'et\'es symplectiques holomorphes irr\'eductibles est plus d\'elicat \`a \'etudier (voir \cite{Huybrechts}). Ce constat est la motivation premi\`ere de cet article, au sens o\`u si $S$ est une surface K3, l'espace de Douady de $n$ points sur $S$, not\'e $S^{[n]}$, est une vari\'et\'e symplectique holomorphe irr\'eductible particuli\`erement bien connue. Dans cet article, on \'etudie ce groupe d'automorphismes pour une surface quelconque (le cas des surfaces K3 est trait\'e dans \cite{BS}).

Dans une premi\`ere partie, on r\'epond \`a quelques questions g\'en\'erales relatives \`a ce groupe d'automorphismes. Principalement, on montre  que :
\begin{itemize}
\item pour toute surface $S$ et tout $n\geq 1$, $\dim\Aut(S^{[n]})=\dim\Aut(S)$;
\item si $S$ est une surface K3 non alg\'ebrique g\'en\'erique, pour tout $n\geq 1$ on a $\Aut(S^{[n]})\cong\Aut(S)$.
\end{itemize}
Ces r\'esultats motivent la seconde partie o\`u l'on \'etudie plus en d\'etail, pour une surface $S$ quelconque, les automorphismes de $S^{[n]}$  provenant d'un automorphisme de $S$ (automorphismes dits \emph{naturels}) : leur action sur la cohomologie, leurs nombres de Lefschetz et la classification de leurs points fixes.

Je remercie Alessandra Sarti de m'avoir incit\'e \`a \'etudier ces questions et pour son soutien lors de la pr\'eparation de cet article, et Arnaud Beauville, Serge Cantat, Antoine Ducros et Manfred Lehn pour les r\'eponses qu'ils m'ont apport\'ees.

\section{Espace de Douady de points sur une surface}

\subsection{Notations et d\'efinitions}

Soit $S$ une surface analytique complexe connexe, compacte et lisse. Pour tout entier $n\geq~0$, notons $S^{(n)}:=S^n/\kS_n$ le quotient
sym\'etrique de $S$, o\`u le groupe sym\'etrique $\kS_n$ agit par permutation des variables, $\pi\colon S^n\to S^{(n)}$ l'application quotient, $\Delta=\bigcup_{i,j}\Delta_{i,j}$ la r\'eunion de toutes les diagonales $\Delta_{i,j}=\{(x_1,\ldots,x_n)\in S^n\,|\,x_i=x_j\}$ et $D:=\pi(\Delta)$ son image dans $S^{(n)}$. La vari\'et\'e $S^{(n)}$ param\`etre les cycles analytiques de dimension z\'ero et de longueur $n$ sur $S$, et est singuli\`ere en chaque point de $\Delta$. Notons $S^{[n]}$ l'\emph{espace de Douady} (\emph{sch\'ema de Hilbert} lorsque $S$ est alg\'ebrique) param\'etrant les sous-espaces analytiques de $S$ de dimension $0$ 
et de longueur $n$. $S^{[n]}$ est une vari\'et\'e analytique complexe compacte lisse de dimension $2n$ (observons que $S^{[0]}$ est un point et $S^{[1]}\cong S$). Le \emph{morphisme de Douady-Barlet} (\emph{de Hilbert-Chow} dans le cas alg\'ebrique) $\rho\colon S^{[n]}\to S^{(n)}$ est projectif et bim\'eromorphe, c'est une r\'esolution des singularit\'es dont nous notons $E:=\rho^{-1}(D)$ le diviseur exceptionnel. Posons $\HilbS:=\coprod_{n\geq 0}S^{[n]}$. Nous nous r\'ef\'erons \`a Grothendieck \cite{Grothendieck} et Fogarty \cite{Fogarty1,Fogarty2} dans le cas alg\'ebrique, \`a Douady \cite{Douady} et de Cataldo{\&}Migliorini \cite{dCM} dans le cas analytique.

\subsection{Automorphismes naturels}

Soit un automorphisme $f\colon\HilbS\xrightarrow{\sim}\HilbS$. Sa restriction \`a chaque composante connexe $S^{[n]}$ est un isomorphisme d'image une composante connexe de m\^eme dimension, donc $f_{|S^{[n]}}=:f_n$ est un automorphisme de $S^{[n]}$. Ainsi, la donn\'ee de $f$ consiste exactement en la donn\'ee d'une famille $(f_n)_{n\geq 0}$ o\`u chaque $f_n$ est un automorphisme de $S^{[n]}$.

Une mani\`ere naturelle de construire un automorphisme de $S^{[n]}$ consiste \`a partir d'un automorphisme de $S$ : si $f\in\Aut(S)$, il induit pour tout $n\geq 1$ un automorphisme de $S^{[n]}$ not\'e $f^{[n]}$ d\'efini par $f^{[n]}(\xi):=f(\xi)$ o\`u $\xi$ est vu \`a gauche comme point de $S^{[n]}$ et \`a droite comme sous-espace analytique de $S$.

Pour $\xi,\xi'\in\HilbS$, la notation $\xi\subset\xi'$ signifie que $\xi$ est un sous-espace analytique de $\xi'$ dans $S$.

\begin{definition}
Un automorphisme $f$ de $\HilbS$ est dit \emph{naturel} s'il v\'erifie pour tous $\xi,\xi'\in\HilbS$ :
$$
\xi\subset\xi' \Longrightarrow f(\xi)\subset f(\xi').
$$
\end{definition}

Notons $\Xi_n\subset S\times S^{[n]}$ la famille universelle, dont les points sont les couples $(x,\xi)\in~S\times S^{[n]}$ tels que $x$ appartient au support de $\xi$ (nous noterons $x\in\xi$). Pour tous entiers $n\geq 0$ et $k\geq 1$, notons $S^{[n,n+k]}$ la sous-vari\'et\'e de $S^{[n]}\times S^{[n+k]}$ dont les points sont les couples $(\xi,\xi')\in S^{[n]}\times S^{[n+k]}$ tels que $\xi\subset\xi'$. Pour $k=1$, $S^{[n,n+1]}$ est lisse de dimension $2n$.

\begin{lemma} Soit $f=(f_n)_{n\geq 0}$ un automorphisme de $\HilbS$. Les assertions suivantes sont \'equivalentes :
\begin{enumerate}
\item\label{item1} $f$ est un automorphisme naturel.
\item\label{item2} $f$ respecte les vari\'et\'es d'incidence : $\forall n,k$, $(f_n\times f_{n+k})(S^{[n,n+k]})=S^{[n,n+k]}$;
\item\label{item3} $f$ respecte les familles universelles : $\forall n$, $(f_1\times f_{n})(\Xi_n)=\Xi_n$;
\item\label{item4} $f$ provient d'un automorphisme de $S$ : $\forall n$, $f_n=f_1^{[n]}$.
\end{enumerate}
\end{lemma}

\begin{proof}\text{}
\begin{itemize}
\item[(\ref{item1})$\Rightarrow$(\ref{item2})] Soit $(\xi,\xi')\in S^{[n,n+k]}$. On a $\xi\subset\xi'$ et puisque $f$ est naturel :
$$f_n(\xi)=f(\xi)\subset f(\xi')=f_{n+k}(\xi'),$$ donc $(f_n\times f_{n+k})(\xi,\xi')=(f_n(\xi),f_{n+k}(\xi'))\in S^{[n,n+k]}$.

\item[(\ref{item2})$\Rightarrow$(\ref{item3})] C'est imm\'ediat puisque $\Xi_n=S^{[1,n]}$.

\item[(\ref{item3})$\Rightarrow$(\ref{item4})] Il suffit de v\'erifier l'\'egalit\'e $f_n=f_1^{[n]}$ sur l'ouvert dense de $S^{[n]}$ constitu\'e des sous-espaces port\'es en $n$ points distincts de $S$. Soit $\xi=\{x_1,\ldots,x_n\}$ un tel sous-espace. Pour tout $i$, $x_i\in\xi$ et la condition $(f_1\times f_n)(\Xi_n)=\Xi_n$ implique $f_1(x_i)\in f_n(\xi)$. Puisque les points $f_1(x_1),\ldots,f_1(x_n)$ sont distincts, on obtient $f_n(\xi)=f_1(\xi)=f_1^{[n]}(\xi)$. 

\item[(\ref{item4})$\Rightarrow$(\ref{item1})] Soit $(\xi,\xi')\in S^{[n,n+k]}$. Avec les notations pr\'ec\'edentes :
$$
f(\xi)=f_n(\xi)=f_1^{[n]}(\xi)=f_1(\xi)\subset f_1(\xi')=f_1^{[n+k]}(\xi')=f_{n+k}(\xi')=f(\xi').
$$
\end{itemize}
\end{proof}

Observons que l'assertion (\ref{item2}) est \'equivalente \`a l'assertion :
\begin{center}
(\ref{item2}') $f$ respecte la vari\'et\'e d'incidence $S^{[n,n+1]}$ : $\forall n$, $(f_n\times f_{n+1})(S^{[n,n+1]})=S^{[n,n+1]}$.
\end{center}
En effet, (\ref{item2})$\Rightarrow$(\ref{item2}') est imm\'ediat et r\'eciproquement, si $(\xi,\xi')\in S^{[n,n+k]}$, on peut construire une tour de sous-espaces $\xi_0=\xi\subset\xi_1\subset\cdots\subset\xi_{k-1}\subset\xi_k=\xi'$ avec $(\xi_i,\xi_{i+1})\in S^{[n+i,n+i+1]}$ pour tout $i=0,\ldots,k-1$. L'assertion (\ref{item2}') appliqu\'ee \`a cette tour fournit :
$$
f_n(\xi)\subset f_{n+1}(\xi_1)\subset\cdots\subset f_{n+k-1}(\xi_{k-1})\subset f_{n+k}(\xi'),
$$
donc $(f_n(\xi),f_{n+k}(\xi'))\in S^{[n,n+k]}$.

\section{Le groupe d'automorphismes de l'espace de Douady de points}

\subsection{Rappel sur les groupes d'automorphismes}

Soit $X$ un espace analytique complexe connexe, compact et lisse. Soit $D_X$ l'espace analytique des sous-espaces analytiques compacts de $X\times X$. D'apr\`es Douady \cite{Douady}, l'espace des applications holomorphes de $X$ dans lui-m\^eme s'identifie -- en prenant le graphe -- \`a un ouvert de $D_X$ dont le groupe $\Aut(X)$ des automorphismes est \`a nouveau un ouvert. D'apr\`es Kerner \cite{Kerner} et Bochner-Montgomery \cite{BM}, $\Aut(X)$ est un groupe de Lie complexe localement compact. D'apr\`es Fujiki \cite{Fujiki}, $X$ ayant une base d\'enombrable d'ouverts, c'est aussi le cas de $D_X$, donc $\Aut(X)$ est d\'enombrable \`a l'infini. En particulier, le nombre de composantes connexes de $\Aut(X)$ est au plus d\'enombrable. L'alg\`ebre de Lie de $\Aut(X)$ s'identifie naturellement \`a l'espace des champs de vecteurs globaux sur $X$, donc $\dim\Aut(X)=\dim H^0(X,T_X)$ (not\'e $h^0(X,T_X)$) o\`u $T_X$ d\'esigne le fibr\'e tangent de $X$, et toutes les composantes connexes ont m\^eme dimension (finie). Si cette dimension est nulle, $\Aut(X)$ est alors totalement discontinu et au plus d\'enombrable.

\subsection{Groupe des automorphismes naturels}

Consid\'erons la famille universelle double $\Xi_n\times\Xi_n\subset S\times S\times S^{[n]}\times S^{[n]}$ (on a permut\'e deux facteurs) :
$$
\Xi_n\times\Xi_n=\{(x,x',\xi,\xi')\,|\,x\in\xi,x'\in\xi'\}
$$
munie des projections $p$ et $q$ vers $S\times S$ et $S^{[n]}\times S^{[n]}$ respectivement. A tout sous-espace analytique compact $\Gamma\subset S\times S$ on associe alors l'espace analytique compact $q(p^{-1}(\Gamma))\subset S^{[n]}\times S^{[n]}$. Cette construction se faisant naturellement en famille, on obtient une application holomorphe $D_S\to D_{S^{[n]}}$. Si $\Gamma$ est le graphe d'un isomorphisme $f$ de $S$,  $q(p^{-1}(\Gamma))$ est le graphe de $f^{[n]}$ donc ce morphisme se restreint en un morphisme de groupes de Lie injectif $\Aut(S)\to\Aut(S^{[n]})$ d'image ferm\'ee puisque ces groupes sont localement compacts (\cite[1.3]{Godement}). Avec un l\'eger abus de langage, nous appellerons ce groupe image le \emph{groupe des automorphismes naturels de $S^{[n]}$}.

\subsection{Dimension du groupe d'automorphismes} 

Pour toute vari\'et\'e analytique complexe lisse compacte $X$ de dimension $d$, notons $\Omega_X:=T_X^\vee$ le fibr\'e vectoriel des formes diff\'erentielles de degr\'e $1$ sur $X$, pour $2\leq p\leq d-1$ notons $\Omega_X^p:=\wedge^p\Omega_X$ et $\omega_X:=\wedge^d\Omega_X$ son fibr\'e canonique. Les nombres de Hodge de $X$ sont par d\'efinition $h^{p,q}(X):=\dim H^q(X,\Omega_X^p)$. Plus g\'en\'eralement, pour tout fibr\'e inversible $L$ sur $X$, nous introduisons les \emph{nombres de Hodge de $X$ \`a valeurs dans $L$} :  $h^{p,q}(X,L):=\dim H^q(X,\Omega^p_X\otimes L)$. Par dualit\'e de Serre, on voit qu'ils v\'erifient :
\begin{align*}
h^{p,q}(X,L)&=h^q(X,\Omega^p_X\otimes L)\\
&=h^{d-q}(X,(\Omega^p_X\otimes L)^\vee\otimes \Omega^d_X)\\
&=h^{d-q}(X,\Omega^{d-p}_X\otimes L^\vee)\\
&=h^{d-p,d-q}(X,L^\vee).
\end{align*}

On observe que $h^0(S^{[n]},T_{S^{[n]}})~=h^{2n}(S^{[n]},\Omega_{S^{[n]}}\otimes \omega_{S^{[n]}})$. 
Rappelons la construction du morphisme de groupes naturel $-_n\colon\Pic(S)\to\Pic(S^{[n]})$. Notons $p_i\colon S^n\to S$ les projections sur les facteurs. Pour tout fibr\'e $L\in\Pic(S)$, le fibr\'e $\bigotimes_{i=1}^np_i^*L$ est naturellement $\kS_n$-lin\'earis\'e et descend en un fibr\'e $\cL\in\Pic(S^{(n)})$ puisque les groupes d'isotropie des points agissent trivialement sur les fibres au-dessus. On d\'efinit alors $L_n:=\rho^*\cL$. On peut montrer que $\omega_{S^{[n]}}=(\omega_S)_n$ (voir par exemple Nieper-Wisskirchen \cite{Nieper}), on a donc $h^0(S^{[n]},T_{S^{[n]}})=h^{2n}(S^{[n]},\Omega_{S^{[n]}}\otimes (\omega_S)_n)$.

La proposition suivante g\'en\'eralise le r\'esultat de G\"ottsche \cite[Proposition 3.3]{Goettsche} relatifs aux nombres de Hodge usuels.

\begin{proposition}
\label{prop:hodgeL}
Pour toute surface analytique compacte lisse $S$ et tout fibr\'e inversible $L$ sur $S$ on a :
$$
\sum_{n=0}^{+\infty}\sum_{p=0}^{2n}h^{p,0}(S^{[n]},L_n)x^p t^n=\frac{(1+xt)^{h^{1,0}(S,L)}}{(1-t)^{h^{0,0}(S,L)}(1-x^2t)^{h^{2,0}(S,L)}}.
$$
\end{proposition}

\begin{proof}
Ce calcul est une g\'en\'eralisation directe de celui de G\"ottsche \cite{Goettsche}, donc nous expliquons simplement les points-cl\'es de l'argument. Notons $S^{(n)}_*:=~S^{(n)}\setminus~D$ l'ouvert lisse et $j\colon S^{(n)}_*\hookrightarrow S^{(n)}$ l'inclusion. Pour $p=0,\ldots,2n$ d\'efinissons $\widetilde{\Omega}^p_{S^{(n)}}:=j_*\Omega^p_{S^{(n)}_*}$. Puisque $\rho\colon S^{[n]}\to S^{(n)}$ est une r\'esolution des singularit\'es, on peut montrer que $\rho_*\Omega^p_{S^{[n]}}=\widetilde{\Omega}^p_{S^{(n)}}$ (voir Steenbrink \cite{Steenbrink}). On a alors~:
$$
H^0(S^{[n]},\Omega^p_{S^{[n]}}\otimes L_n)\cong H^0(S^{(n)}_*,\Omega^p_{S^{(n)}_*}\otimes \cL)\cong H^0\left(S^n,\Omega^p_{S^n}\otimes\bigotimes_{i=1}^np_i^*L\right)^{\kS_n}.
$$
Par la formule de K\"unneth, $H^0\left(S^n,\Omega^*_{S^n}\otimes\bigotimes_{i=1}^np_i^*L\right)\cong\bigotimes_{i=1}^np_i^*H^0(S,\Omega^*_S\otimes L)$. Soit $\omega_1,\ldots,\omega_m$ une base homog\`ene de $H^0(S,\Omega^*_S\otimes L)$ : pour tout $i=1,\ldots,m$ il existe $d_i\in~\{0,1,2\}$ tel que $\omega_i\in H^0(S,\Omega^{d_i}_S\otimes L)$. Une base de $H^0\left(S^n,\Omega^*_{S^n}\otimes\bigotimes_{i=1}^np_i^*L\right)^{\kS_n}$ est donc form\'ee des $\sum_{\sigma\in\kS_n}\sigma^*\eta$ pour $\eta$ de la forme $p_1^*\omega_{i_1}\wedge\ldots\wedge p_n^*\omega_{i_n}$ avec $\sum_{j=1}^n d_{i_j}=~p$. Un peu de combinatoire permet alors de conclure.
\end{proof}

\begin{corollaire}
\label{cor:dimension}
Pour toute surface analytique compacte lisse $S$ et tout $n\geq 1$  on a $\dim\Aut(S^{[n]})=\dim\Aut(S)$.
\end{corollaire}

\begin{proof}
Observons que 
$$
h^0(S^{[n]},T_{S^{[n]}})=h^{1,2n}(S^{[n]},(\omega_S)_n)=h^{2n-1,0}(S^{[n]},((\omega_S)_n)^\vee)=h^{2n-1,0}(S^{[n]},(\omega_S^\vee)_n).
$$
Un examen attentif de la formule donn\'ee dans la proposition \ref{prop:hodgeL} fournit :
$$
h^{2n-1,0}(S^{[n]},L_n)=h^{1,0}(S,L)\binom{n-1}{h^{2,0}(S,L)+n-2}.
$$
Or, on calcule que $h^{1,0}(S,\omega_S^\vee)=h^0(S,T_S)$ et $h^{2,0}(S,\omega_S^\vee)=h^{2,2}(S)=1$, ce qui donne $h^0(S^{[n]},T_{S^{[n]}})=h^0(S,T_S)$.
\end{proof}

Ce r\'esultat implique que les composantes connexes de l'identit\'e de $\Aut(S^{[n]})$ et $\Aut(S)$ sont isomorphes, et de fa\c{c}on intuitive : il n'y a ``pas trop'' d'automorphismes non naturels. Il est notable que la dimension de $\Aut(S^{[n]})$ ne d\'epende pas de $n$ !

\begin{remark}
Si $S$ est un tore ou une surface K3, le faisceau canonique $\omega_S$ est trivial et $S^{[n]}$ admet une structure symplectique \cite[proposition 5]{Beauvillec1nul} donc $T_{S^{[n]}}\cong~\Omega_{S^{[n]}}$ et $h^0(S^{[n]},T_{S^{[n]}})=h^{1,0}(S^{[n]})$. La formule de G\"ottsche cit\'ee plus haut donne imm\'ediatement  $h^0(T_{S^{[n]}})=h^{1,0}(S)=h^0(T_S)$ (qui vaut $2$ pour un tore complexe et $0$ pour une surface K3).
\end{remark}

La proposition \ref{prop:hodgeL} sugg\`ere la conjecture suivante g\'en\'eralisant la formule donnant les nombres de Hodge usuels conjectur\'ee par G\"ottsche \cite{Goettsche}, d\'emontr\'ee par G{\"o}ttsche\&Soergel \cite{GS} (voir aussi Cheah \cite{Cheah}) dans le cas alg\'ebrique et g\'en\'eralis\'ee par de Cataldo\&Migliorini \cite{dCM} aux surfaces analytiques k\"ahl\'eriennes :

\begin{conjecture}
Pour toute surface analytique k\"ahl\'erienne compacte lisse $S$ et tout fibr\'e inversible $L$ sur $S$ on a :
$$
\sum_{n=0}^\infty h^{p,q}(S^{[n]},L_n)x^py^qt^n=\prod_{k=1}^\infty\prod_{p=0}^2\prod_{q=0}^2 \left(\frac{1}{1-(-1)^{p+q}x^{p+k-1}y^{q+k-1}t^k}\right)^{(-1)^{p+q}h^{p,q}(S,L)}
$$
\end{conjecture}

\subsection{Cas particulier des surfaces K3}

Si $S$ est une surface K3, le corollaire~\ref{cor:dimension} dit que les automorphismes \emph{non naturels} $\Aut(S^{[n]})\setminus \Aut(S)$ sont au plus d\'enombrables. On en conna\^it cependant tr\`es peu, il semble que le seul exemple connu soit celui-ci : soit $S\in\IP^3_\IC$ une quartique g\'en\'erique ne contenant aucune droite. Pour tout point $\xi\in S^{[2]}$ il existe une unique droite de $\IP^3_\IC$ contenant $\xi\subset S$, qui recoupe $S$ en deux autres points, ce qui d\'efinit une involution birationnelle sur $S^{[2]}$ qui s'\'etend en un isomorphisme (voir \cite{Beauvillec1nul}). 

Beauville \cite{BeauvilleKaehler} montre que si $S$ est une surface simplement connexe, de groupe de Picard nul et de groupe d'automorphismes trivial, alors pour tout $n$ le groupe $\Aut(S^{[n]})$ est aussi trivial. Ce r\'esultat se g\'en\'eralise :

\begin{proposition}
\label{prop:K3}
Si $S$ est simplement connexe, de groupe de Picard nul et telle que $h^0(S,T_S)=0$, alors $\Aut(S)\cong \Aut(S^{[n]})$ pour tout $n\geq 1$.
\end{proposition}

La condition de non existence de sections globales non nulles du fibr\'e tangent signifie que le groupe $\Aut(S)$ est totalement discontinu. S'il est en particulier trivial, nous retrouvons le th\'eor\`eme de Beauville cit\'e ci-dessus.

\begin{proof} La d\'emonstration reprend au d\'ebut celle de Beauville (\cite[\S5]{BeauvilleKaehler}); nous en rappelons les arguments pour faciliter la lecture. 

Soit $n\geq 2$ et $f$ un automorphisme de $S^{[n]}$. Il induit un automorphisme de son groupe de Picard. Puisque $\Pic(S)=\{0\}$, $\Pic(S^{[n]})$ est engendr\'e par un fibr\'e $L$ tel que $L^{\otimes 2}\cong \cO(E)$, donc $f$ pr\'eserve $E$  (car $E$ est rigide). $f$ induit donc un automorphisme de $S^{[n]}\setminus E$. La compos\'ee du quotient par le groupe sym\'etrique $S^n\setminus\Delta\to~S^{(n)}\setminus~D$ avec l'isomorphisme $S^{(n)}\setminus D\cong S^{[n]}\setminus E$ est un rev\^etement fini et non ramif\'e $\pi\colon~S^n\setminus~\Delta\to~S^{[n]}\setminus E$. Puisque $S$ est simplement connexe est que $\Delta$ est de codimension $2$ dans $S^n$, $S^n\setminus \Delta$ est aussi simplement connexe donc $\pi$ est le rev\^etement universel de $S^{[n]}\setminus E$. L'automorphisme $f$ s'y remonte donc en un automorphisme $\bar{f}$ de $S^n\setminus\Delta$.

Soient $s_1,\ldots,s_n$ des points de $S$ deux \`a deux distincts. Consid\'erons le morphisme compos\'e $S\setminus\{s_2,\ldots,s_n\}\to S^n\setminus\Delta\to S^n\setminus\Delta$ d\'efini par
$x_1\mapsto \bar{f}(x_1,s_2,\ldots,s_n)$. L'une au moins de ses compos\'ees \`a droite avec les projections $p_1,\ldots,p_n$ sur les diff\'erents facteurs est non constante puisque $\bar{f}$ est injective, et quitte \`a permuter les variables on peut supposer que $p_1$ a cette propri\'et\'e. L'application rationnelle $x_1\mapsto p_1\circ\bar{f}(x_1,s_2,\ldots,s_n)$ \'etant non constante, c'est en fait un isomorphisme de $S$ (\cite[Lemme 1]{BeauvilleKaehler}) qui d\'epend contin\^ument de $s_2,\ldots,s_n$ dans un voisinage de ces points o\`u cette compos\'ee reste un isomorphisme. Mais le groupe d'automorphismes de $S$ \'etant totalement discontinu, cet isomorphisme est ind\'ependant localement de $s_2,\ldots,s_n$ : notons-le $\phi_1(x_1)$. Continuons avec la deuxi\`eme coordonn\'ee en consid\'erant maintenant la compos\'ee $x_2\mapsto \bar{f}(s_1,x_2,s_3,\ldots,s_n)$. La premi\`ere projection sur $S$ est constante de valeur $\phi_1(s_1)$, donc l'une des projections $p_2,\ldots,p_n$ est non constante, et nous pouvons supposer qu'il s'agit de $p_2$. Comme pr\'ec\'edemment, $x_2\mapsto p_2\circ\bar{f}(s_1,x_2,s_3,\ldots,s_n)$ est en fait un isomorphisme de $S$ qui ne d\'epend pas de $s_1,s_3,\ldots,s_n$ dans un voisinage de ces points, et que nous notons $\phi_2(x_2)$. On traite ainsi de proche en proche toutes les coordonn\'ees. Finalement, dans un voisinage de $(s_1,\ldots,s_n)\in S^n\setminus\Delta$ l'automorphisme $\bar{f}$ co\"incide avec l'application $(x_1,\ldots,x_n)\mapsto(\phi_1(x_1),\ldots,\phi_n(x_n))$. Puisque $S^n\setminus\Delta$ est connexe, ces deux fonctions sont \'egales en tout point.

Ceci fait, l'automorphisme $\bar{f}(x_1,\ldots,x_n)=(\phi_1(x_1)\ldots,\phi_n(x_n))$ de $S^n\setminus\Delta$ doit redescendre \`a $S^{(n)}\setminus D$ pour donner $f$. Le rev\^etement \'etant galoisien de groupe $\kS_n$, c'est le cas si et seulement s'il existe
un isomorphisme $\psi$ du groupe $\kS_n$ tel que $\bar{f}(\sigma\cdot(x_1,\ldots,x_n))=\psi(\sigma)\cdot(\phi_1(x_1)\ldots,\phi_n(x_n))$ pour tout $\sigma\in\kS_n$. Si $\psi$ est diff\'erent de l'identit\'e, prenons $\sigma$ tel que $\psi(\sigma)\neq\sigma$ et trois entiers $i,j,k$ avec $\sigma(i)=j$ et $k=\psi(\sigma)(i)\neq j$. La condition d'\'equivariance dit que $\phi_i(x_j)=\phi_k(x_k)$ pour tout $x_j\neq x_k$ et contredit le fait que $\bar{f}$ est \`a valeurs dans $S^n\setminus\Delta$. Donc $\psi$ est l'identit\'e et la condition d'\'equivariance dit que pour tous indices $i,j$ on a $\phi_i(x_i)=\phi_j(x_j)$ pour $x_i\neq x_j$. Cette \'egalit\'e reste vraie pour $x_i=x_j$ donc $\phi_i=\phi_j$ pour tous $i,j$. En notant $\phi$ cet isomorphisme de $S$, on a donc $\bar{f}(x_1,\ldots,x_n)=(\phi(x_1),\ldots,\phi(x_n))$ et $\bar{f}$ redescend. L'isomorphisme $f$ de $S^{[n]}$ de d\'epart provient donc de l'isomorphisme $\phi$ de $S$. Le morphisme injectif de groupes de Lie $\Aut(S)\to\Aut(S^{[n]})$ est donc aussi surjectif, donc c'est un isomorphisme puisque $\Aut(S)$ est localement compact et d\'enombrable \`a l'infini \cite[1.10 Corollaire 1]{Godement}.
\end{proof}

Les hypoth\`eses de cette proposition ne sont remplies que par des surfaces K3 non alg\'ebriques, auquel cas le groupe de Picard est g\'en\'eriquement nul, donc on a imm\'ediatement :

\begin{corollaire}
Les automorphismes de l'espace de Douady de points sur une surface K3 g\'en\'erique sont tous naturels.
\end{corollaire}

\begin{remark}
En utilisant le th\'eor\`eme de Torelli global, McMullen \cite{McMullen} construit des surfaces K3 non alg\'ebriques de groupe de Picard nul ayant un groupe d'automorphismes d'ordre infini.
\end{remark}

\section{Etude cohomologique}

\subsection{Pr\'eliminaires sur les traces}

Soit $V$ un espace vectoriel de dimension finie sur $\IC$ et $\kf\in\End(V)$. Notons $\trace_V(\kf)\in\IC$ sa trace et $\chi_V(\kf):=\det(\kf-X \id_V)\in\IC[X]$ son polyn\^ome caract\'eristique. Si $W$ et $\kg\in\End(W)$ sont donn\'es similairement, on a les formules :
\begin{equation}
\label{formule:trace1}
\begin{split}
\trace_{V\oplus W}(\kf\oplus \kg)&=\trace_V(\kf)+\trace_W(\kg)\in\IC,\\
\trace_{V\otimes W}(\kf\otimes \kg)&=\trace_V(\kf)\cdot\trace_W(\kg)\in\IC.
\end{split}
\end{equation}
Supposons maintenant que $V=\bigoplus_{n=0}^{+\infty} V_n$ est un espace vectoriel gradu\'e, chaque $V_n$ \'etant de dimension finie, et que $\kf=(\kf_n)_{n\geq 0}$ avec $\kf_n\in\End(V_n)$. On d\'efinit sa trace \`a valeurs dans $\IC[[q]]$ par :
$$
\trace_V(\kf):=\sum_{n=0}^{+\infty}\trace_{V_n}(\kf_n)q^n.
$$
Si $W=\bigoplus_{n=0}^{+\infty} W_n$ et $\kg=(\kg_n)_{n\geq 0}$ sont donn\'es similairement, les formules (\ref{formule:trace1}) s'\'etendent, \`a valeurs cette fois dans $\IC[[q]]$.

Soit $E$ un espace vectoriel de dimension finie sur $\IC$. Notons  son alg\`ebre tensorielle $T(E):=~\bigoplus_{n=0}^{+\infty}T^n(E)$, son alg\`ebre sym\'etrique $S(E):=\bigoplus_{n=0}^{+\infty}S^n(E)$  et son alg\`ebre altern\'ee $\Lambda(E):=~\bigoplus_{n=0}^{+\infty}\Lambda^n(E)$. Pour  $\kf\in\End(E)$, on calcule ais\'ement les traces des endomorphismes induits sur ces trois alg\`ebres :
\begin{equation}
\label{formule:trace2}
\begin{split}
\trace_{T(E)}(T(\kf))&=\frac{1}{1-\trace_E(\kf)q}\in\IC[[q]],\\
\trace_{S(E)}(S(\kf))&=\frac{1}{(-q)^{\dim E}\chi_E(\kf)(1/q)}\in\IC[[q]],\\
\trace_{\Lambda(E)}(\Lambda(\kf))&=q^{\dim E}\chi_E(\kf)(-1/q)\in\IC[q].
\end{split}
\end{equation}
En effet, il suffit de consid\'erer le cas o\`u $\kf$ est diagonalisable et conclure par densit\'e, et sous cette hypoth\`ese, si l'on note $\chi_E(\kf)=(-1)^{\dim E}\prod_{i=1}^{\dim E}(X-\lambda_i)$, alors :
\begin{align*}
(-q)^{\dim E}\chi_E(\kf)(1/q)=\prod_{i=1}^{\dim E}(1-\lambda_iq)\\
q^{\dim E}\chi_E(\kf)(-1/q)=\prod_{i=1}^{\dim E}(1+\lambda_iq).
\end{align*}

Soit $A=\bigoplus_{i=0}^d A^i$ un espace vectoriel gradu\'e de dimension finie sur $\IC$ et $\kf=~(\kf_i)_{i=0,\ldots,d}$ avec $\kf_i\in\End(A^i)$. Posons :
\begin{align*}
\trace_A(\kf)&=\sum_{i=0}^d\trace_{A^i}(\kf_i)t^i\in\IC[t],\\
\chi_A(\kf)&=\prod_{i=0}^dt^{i\dim A^i}\chi_{A^i}(\kf)(X/t^i)\in\IC[t][X].
\end{align*}
Les alg\`ebres $T(A)$, $S(A)$ et $\Lambda(A)$ sont munies d'une double graduation en poids $n$ et degr\'e $i$, l'alg\`ebre $A$ \'etant prise de poids $1$. Les formules (\ref{formule:trace2}) s'\'etendent alors, \`a valeurs dans $\IC[t][[q]]$ et $\IC[t][q]$ respectivement. En effet, lorsque $\kf$ est diagonalisable, le polyn\^ome caract\'eristique ainsi d\'efini tient compte du degr\'e des vecteurs propres en ins\'erant $t^i$ devant les valeurs propres $\lambda_{i,j}$ de $\kf_i$ :
$$
\chi_A(\kf)=\prod_{i=0}^d(-1)^{\dim A^i}\prod_{j=1}^{\dim A^i}(X-\lambda_{i,j}t^i).
$$

\subsection{Trace d'un op\'erateur naturel sur un espace de Fock}

Soit $A$ une super-alg\`ebre de Frob\'enius gradu\'ee de dimension finie sur $\IC$, not\'ee $A=\bigoplus_{i=0}^{4n}A^i$. Elle est munie :
\begin{itemize}
\item d'une super-structure (\ie graduation par $\IZ/2\IZ$) d\'efinie par la d\'ecomposition $A^{\text{pair}}:=\bigoplus_{i=0}^{2n}A^{2i}$ et $A^{\rm{impair}}:=\bigoplus_{i=0}^{2n-1}A^{2i+1}$;
\item d'une multiplication gradu\'ee commutative et associative $A\otimes A\to A$ : si $a\in A^i$ et $b\in A^j$, alors $a\cdot b\in A^{i+j}$ et $a\cdot b=(-1)^{i\cdot j} b\cdot a$;
\item d'une forme lin\'eaire $T\colon A\to\IC$ de degr\'e $-4n$ telle que la forme bilin\'eaire super-sym\'etrique $\langle a,b\rangle:=T(a\cdot b)$ est non d\'eg\'en\'er\'ee : la super-sym\'etrie s'\'ecrit $\langle a,b\rangle=~(-1)^{|a|\cdot |b|}\langle b,a\rangle$, o\`u $|\cdot|$ d\'esigne le degr\'e d'un \'el\'ement homog\`ene.
\end{itemize}

Consid\'erons l'alg\`ebre de Heisenberg $\kh_A:=A[t,t^{-1}]\oplus\IC \kappa$ o\`u le crochet de Lie est d\'efini par :
\begin{align*}
[\kappa,-]&=0 \text{ ($\kappa$ est central)},\\
[a\otimes\phi,b\otimes\psi]&=-\res_{t=0}(\phi\;d\psi)\cdot \langle a,b\rangle\cdot \kappa,
\end{align*}
pour $a,b\in A$ et $\phi,\psi\in\IC[t,t^{-1}]$. Ce crochet est super-antisym\'etrique au sens o\`u l'on a : $[a\otimes\phi,b\otimes\psi]=-(-1)^{|a|\cdot |b|}[b\otimes\psi,a\otimes\phi]$. Pour $a\in A$ et $n\in\IZ$, notons $a_n:=a\otimes t^n$. Le crochet de Lie est enti\`erement caract\'eris\'e par les relations :
\begin{align*}
[\kappa,a_n]&=0,\\
[a_n,b_m]&=n\cdot\delta_{n,-m}\cdot\langle a,b\rangle\cdot \kappa
\end{align*}
o\`u $\delta$ d\'esigne le symbole de Kronecker.

Soit $I\subset U(\kh_A)$ l'id\'eal \`a gauche dans l'alg\`ebre enveloppante de $\kh_A$ engendr\'e par les \'el\'ements de la forme $a_n$ pour $n\leq 0$ et l'\'el\'ement $\kappa-1$. L'\emph{espace de Fock} est par d\'efinition le quotient (\`a droite) $\IF(A):=U(\kh_A)/I$. C'est une repr\'esentation de $\kh_A$, engendr\'ee par la classe not\'ee $\vac$ de $1\in U(\kh_A)$. 
Notons, pour cette repr\'esentation $\kh_A\to\End(\IF(A))$, $\kq_n(a)$ l'endomorphisme correspondant \`a l'\'element $a_n$. On a donc :
\begin{align*}
\kq_n(a)\vac&=0\text{ si } n\leq 0,\\
\kappa\vac&=\vac,\\
[\kq_n(a),\kq_m(b)]&=\kq_n(a)\circ\kq_m(b)-(-1)^{|a|\cdot|b|}\kq_m(b)\circ\kq_n(a)\\
&=n\delta_{n,-m}\cdot\langle a,b\rangle\cdot \id_{\IF(A)}.
\end{align*}
On constate ais\'ement que cette repr\'esentation est irr\'eductible (voir \cite{LehnMontreal}).

On munit $\IF(A)$ d'une double graduation par poids et degr\'e en d\'eclarant qu'un endomorphisme $\kq_n(a)$ est de poids $n$ et de degr\'e $2(n-1)+|a|$, et que $\vac$ est de poids et degr\'e nuls. La d\'ecomposition en poids est alors not\'ee $\IF(A)=\bigoplus_{n=0}^{+\infty}A^{[n]}$. 

\begin{remark}
Lehn\&Sorger \cite{LS2} ont construit sur $A^{[n]}$ une structure de super-alg\`ebre de Frob\'enius gradu\'ee. Nous n'utiliserons pas cette structure suppl\'ementaire dans la suite.
\end{remark}

Pour $n,m\geq 0$, on a $\kq_n(a)\circ\kq_m(b)=(-1)^{|a|\cdot|b|}\kq_m(b)\circ\kq_n(a)$ donc l'alg\`ebre $\IF(A)$ est sym\'etrique en ce qui concerne $A^{\text{pair}}$ et antisym\'etrique pour $A^{\rm{impair}}$. En notant $Aq^m$ l'alg\`ebre $A$ consid\'er\'ee de poids $m$ (au lieu de $1$) on a donc :
$$
\IF(A)\cong S^{\star}\left(\bigoplus_{m\geq 1}Aq^m\right)\cong \bigotimes_{m\geq 1}S(A^{\text{pair}}q^m)\otimes\bigotimes_{m\geq 1}\Lambda(A^{\rm{impair}}q^m)
$$
o\`u $S^\star$ d\'esigne le produit super-sym\'etrique (\ie le produit sym\'etrique sur un $\IZ/2\IZ$-espace vectoriel).

Soit $\kf=(\kf_i)_{i=0,...4n}$ avec $\kf_i\in\End(A^i)$ un endomorphisme de $A$. Il induit un endomorphisme de $\IF(A)$ not\'e $\IF(\kf)$, caract\'eris\'e par la formule :
$$
\IF(\kf)\left(\kq_{\lambda_1}(a_1)\circ\cdots\circ\kq_{\lambda_k}(a_k)\vac\right)=\kq_{\lambda_1}(\kf(a_1))\circ\cdots\circ\kq_{\lambda_k}(\kf(a_k))\vac
$$
o\`u $\kf(a_i)$ signifie bien s\^ur $\kf_{|a_i|}(a_i)$. Les formules de trace pr\'ec\'edentes permettent le calcul de la trace de $\IF(\kf)$ :
\begin{proposition}
\label{prop:traceFock}
$$
\trace_{\IF(A)}(\IF(\kf))=\prod_{m\geq 1}\frac{q^{m\cdot\dim A^{\rm{impair}}}\chi_{A^{\rm{impair}}}(\kf)(-1/q^m)}{(-q^{m})^{\dim A^{\rm{pair}}}\chi_{A^{\rm{pair}}}(\kf)(1/q^m)}\in\IC[t][[q]].
$$
\end{proposition}
Remarquons que si $\kf$ est l'identit\'e, on retrouve la formule d\'ej\`a connue de la s\'erie de Poincar\'e de l'espace de Fock.

\subsection{Rappels sur la cohomologie de $\HilbS$} Les r\'esultats rappel\'es ici sont dus \`a G\"ottsche \cite{Goettsche} et Nakajima \cite{Nakajima} pour les surfaces alg\'ebriques, \'etendus aux surfaces analytiques par de Cataldo\&Migliorini \cite{dCM}.

Soit $H^*(S^{[n]}):=\bigoplus_{i=0}^{4n} H^i(S^{[n]})$ l'alg\`ebre de cohomologie singuli\`ere \`a coefficients complexes de $S^{[n]}$, la structure d'anneau \'etant donn\'ee par le cup-produit. \emph{L'alg\`ebre de cohomologie totale des espaces de Douady de points sur $S$} est $\IH_S:=\bigoplus_{n\geq 0} H^*(S^{[n]})$. L'unit\'e de $H^*(S^{[0]})\cong\IC$ est appel\'ee le \emph{vacuum} et est not\'ee $\vac$.

L'espace $\IH_S$ est muni d'une double graduation : les \'el\'ements de $H^i(S^{[n]})$ sont dits de bidegr\'e $(n,i)$ o\`u $n$ est le \emph{poids conforme} et $i$ le \emph{degr\'e cohomologique}. Un op\'erateur lin\'eaire $\kg\in\End(\IH_S)$ est dit \emph{homog\`ene} de bidegr\'e $(u,v)$ s'il v\'erifie pour tous $n,i$ : $\kg(H^i(S^{[n]}))\subset H^{i+v}(S^{[n+u]})$. On notera $v=:|\kg|$ le degr\'e cohomologique de $\kg$. Le commutateur de deux op\'erateurs lin\'eaires homog\`enes $\kg_1,\kg_2$ est d\'efini par $[\kg_1,\kg_2]=\kg_1\circ\kg_2 -(-1)^{|\kg_1|\cdot|\kg_2|}\kg_2\circ\kg_1$.

Le produit d'intersection sur $H^*(S^{[n]}])$ d\'efini par $\langle\alpha,\beta\rangle_n:=\int_{S^{[n]}}\alpha\beta$ s'\'etend naturellement en une forme bilin\'eaire sym\'etrique non d\'eg\'en\'er\'ee not\'ee $\langle\cdot,\cdot\rangle$ sur $\IH_S$. Tout op\'erateur lin\'eaire homog\`ene $\kg\in\End(\IH_S)$ admet donc un \emph{adjoint} not\'e $\kg^\dagger$ caract\'eris\'e par la relation $\langle \kg(\alpha),\beta\rangle=(-1)^{|\kg|\cdot|\alpha|} \langle\alpha,\kg^\dagger(\beta)\rangle$.

Soit $\kh_{H^*(S)}:=H^*(S)[t,t^{-1}]\oplus \IC c$ munie comme pr\'ec\'edemment du crochet de Lie pour lequel $c$ est central et $[\alpha\,f(t),\beta\,g(t)]=\int_S\alpha\beta\cdot \res_tg\,df\cdot c$
pour $\alpha,\beta\in H^*(S)$ et $f,g\in\IC[t,t^{-1}]$. On construit g\'eom\'etriquement une repr\'esentation irr\'eductible $\kh_{\IH^*(S)}\rightarrow\End(\IH_S)
$ au moyen des op\'erateurs de Nakajima dont nous rappelons la d\'efinition pour un usage ult\'erieur.

Pour tous $n\geq 0$ et $k\geq 1$, soit $\Sigma^{[n,n+k]}\subset S^{[n]}\times S\times S^{[n+k]}$ la sous-vari\'et\'e dont les points sont les triplets $(\xi,x,\xi')$ tels que $\xi\subset\xi'$ et le support de $\cI_\xi/\cI_{\xi'}$ est $\{x\}$, o\`u $\cI_\xi$ d\'esigne le faisceau d'id\'eaux du sous-espace $\xi$ de $S$. Notons les diff\'erentes projections sur les facteurs ainsi :
$$
\xymatrix{& S^{[n]}\times S\times
S^{[n+k]}\ar[dl]_\varphi\ar[d]^\rho\ar[dr]^\psi\\S^{[n]}& S& S^{[n+k]}}
$$
Les \emph{op\'erateurs de Nakajima} $\kq_k:H^*(S)\rightarrow \End(\IH_S)$ pour $k\in \IZ$
sont d\'efinis ainsi : pour tous $k\geq 0$, $\alpha\in H^*(S)$ et $x\in H^*(S^{[n]})$ on pose
$$
\kq_k(\alpha)(x):=\psi_*\left(\varphi^*(x)\cdot \rho^*(\alpha)\cdot\left[\Sigma^{[n,n+k]}\right]\right)
$$
o\`u $\psi_*$ d\'esigne l'image directe de cohomologie singuli\`ere d\'efinie \`a partir de l'image directe d'homologie
en utilisant la dualit\'e de Poincar\'e et $\left[\Sigma^{[n,n+k]}\right]$ d\'esigne la classe fondamentale cohomologique de la sous-vari\'et\'e.
Les op\'erateurs d'indice n\'egatif sont ensuite d\'efinis par adjonction $\kq_{-k}(\alpha):=(-1)^k \kq_k(\alpha)^\dagger$ pour $k>0$.
On convient de poser $\kq_0=0$. Les op\'erateurs $\kq_k$ sont appel\'es \emph{op\'erateurs de cr\'eation}
si $k\geq 1$ et \emph{op\'erateurs d'annihilation} si $k\leq -1$.
Le fait que ces op\'erateurs fournissent une repr\'esentation de $\kh_{H^*(S)}$ r\'esulte du th\'eor\`eme de Nakajima \cite{Nakajima} donnant leur r\`egle de commutation : $[\kq_i(\alpha),\kq_j(\beta)]=i\cdot\delta_{i+j,0}\cdot \int_S\alpha\beta\cdot\id_{\IH_S}$. Le fait que la repr\'esentation $\kh_{H^*(S)}\rightarrow\End(\IH_S)$ soit irr\'eductible r\'esulte de l'\'egalit\'e entre la s\'erie de Poincar\'e de $\IH_S$ et celle de la repr\'esentation irr\'eductible de $\kh_{H^*(S)}$ donn\'ee par l'espace de Fock. Autrement dit, $\IH_S\cong\IF(H^*(S))$ comme repr\'esentations de $\kh_{H^*(S)}$, ce qui se traduit par le fait que l'espace $\IH_S$ admet une base constitu\'ee des vecteurs de la forme :
$$
\kq_{n_1}(u_1)\cdots\kq_{n_k}(u_k)\vac
$$
pour $k\geq 0$ et $n_i\geq 1$, o\`u les $u_i$ parcourent une base de $H^*(S)$ (en faisant attention aux signes car si un tel vecteur contient deux fois le m\^eme $u_j$ de degr\'e cohomologique impair, alors ce vecteur est nul en raison de l'intervention des signes dans la r\`egle de commutation des op\'erateurs de Nakajima).

\subsection{Action des automorphismes naturels sur la cohomologie}

Soit $f\in~\Aut(S)$ et $(f_n)_{n\geq 0}$ l'automorphisme naturel de $\HilbS$ induit (avec $f_n=f^{[n]}$). Cela induit un op\'erateur lin\'eaire sur $\IH_S$ de poids conforme nul dont les composantes sont not\'ees $f_n^*\in\End(H^*(S^{[n]}))$. Notons $f^\star$ l'op\'erateur sur $\IH_S$ tel que $f^\star_{|H^*(S^{[n]})}=f_n^*$, avec $f_1^*=f^*$. La relation entre $f^\star$ et les op\'erateurs de Nakajima est donn\'ee par la formule suivante :
\begin{lemma} Pour tous $k\in\IZ$ et $\alpha\in H^*(S)$ on a :
$$
f^\star\circ\kq_k(\alpha)=\kq_k\left(f^*\alpha\right)\circ f^\star.
$$
\end{lemma}

\begin{proof} Commen\c{c}ons par le cas $k\geq 1$. Soit $n\geq 0$ et $x\in H^*(S^{[n]})$. On calcule :
$$
f_{n+k}^*\kq_k(\alpha)x=f_{n+k}^*\psi_*\left(\varphi^*x\cdot\rho^*\alpha\cdot\left[\Sigma^{[n,n+k]}\right]\right).
$$
Le diagramme suivant est commutatif :
$$
\xymatrix{S^{[n]}\times S\times S^{[n+k]}\ar[rr]^{f_n\times f\times f_{n+k}}\ar[d]^\psi&&S^{[n]}\times S\times S^{[n+k]}\ar[d]^\psi\\S^{[n+k]}\ar[rr]^{f_{n+k}}&&S^{[n+k]}}
$$
et on a la formule: $f_{n+k}^*\circ\psi_*=\psi_*\circ (f_n\times f\times f_{n+k})^*$. Observons aussi que $\varphi\circ~(f_n\times~f\times~f_{n+k})=f_n\circ\varphi$ et  $\rho\circ(f_n\times f\times f_{n+k})=f\circ\rho$ et nous obtenons~:
$$
f_{n+k}^*\kq_k(\alpha)x=\psi_*\left(\varphi^*f_n^*x\cdot\rho^*f^*\alpha\cdot\left[(f_n\times f\times f_{n+k})^{-1}(\Sigma^{[n,n+k]})\right]\right).
$$ 
Puisque $(f_n)_{n\geq 0}$ est naturel, $(f_n\times f\times f_{n+k})^{-1}(\Sigma^{[n,n+k]})=\Sigma^{[n,n+k]}$ donc finalement~:
\begin{align*}
f_{n+k}^*\kq_k(\alpha)x&=\psi_*\left(\varphi^*f_n^*x\cdot\rho^*f^*\alpha\cdot\left[\Sigma^{[n,n+k]}\right]\right)\\
&=\kq_k(f^*\alpha)f_n^*(x).
\end{align*}

Pour traiter le cas $k<0$, proc\'edons par adjonction. Pour $k>0$, nous avons obtenu la formule $f^\star\circ\kq_k(\alpha)=\kq_k(f^*\alpha)\circ f^\star$. Par adjonction nous obtenons alors :
$$
\kq_{-k}(\alpha)\circ (f^\star)^\dagger=(f^\star)^\dagger\circ\kq_{-k}(f^*\alpha).
$$
Puisque $(f^\star)^\dagger=(f^{-1})^\star$, on obtient la formule indiqu\'ee.
\end{proof}

Ce lemme implique que pour toute classe $\kq_{n_1}(u_1)\cdots\kq_{n_k}(u_k)\vac\in H^*(S^{[n]})$ on a :
$$
f_n^*\left(\kq_{n_1}(u_1)\cdots\kq_{n_k}(u_k)\vac\right)=\kq_{n_1}(f^*u_1)\cdots\kq_{n_k}(f^*u_k)\vac.
$$
Par ailleurs, partant de l'automorphisme $f$ de $S$, si nous notons $\kf:=f^*$ l'endomorphisme de $H^*(S)$ induit par $f$, nous pouvons \'etendre $\kf$ en un endomorphisme de $\IH_S$ en utilisant la base de la cohomologie donn\'ee par les classes form\'ees d'op\'erateurs de Nakajima. Nous d\'efinissons ainsi un op\'erateur lin\'eaire $\kf^{[n]}$ pour tout $n$ par :
$$
\kf^{[n]}\left(\kq_{n_1}(u_1)\cdots\kq_{n_k}(u_k)\vac\right):=\kq_{n_1}(\kf u_1)\cdots\kq_{n_k}(\kf u_k)\vac
$$
Avec cette d\'efinition, on obtient la formule attendue $\kf^{[n]}=f_n^*$, ou de fa\c{c}on plus sym\'etrique : $(f^*)^{[n]}=(f^{[n]})^*$, ou encore $f^\star=\IF(f^*)$.

\begin{remark}
Soit $f\in\Aut(S)$. Les nombres de Lefschetz de l'automorphisme naturel induit sur $S^{[\bullet]}$ se calculent donc avec la proposition \ref{prop:traceFock}.
\end{remark}

Pour toute vari\'et\'e analytique connexe, compacte et lisse $X$ et tout automorphisme $f$ de $X$, on note $\rho(f)$ le \emph{rayon spectral} de $f^*$ sur $H^*(X,\IC)$ et $\ent(f):=\log\rho(f)$ son \emph{entropie}.

\begin{corollaire}
Soit $f\in\Aut(S)$. Pour tout $n\geq 1$ on a $\ent(f^{[n]})=n\cdot\ent(f)$.
\end{corollaire}

\begin{proof}
Il suffit de le montrer lorsque $f^*$ est diagonalisable. Si $u_1,\ldots,u_k$ est une base de diagonalisation de $f^*$, de valeurs propres $\lambda_1,\ldots,\lambda_k$, on a :
$$
(f^{[n]})^*\left(\kq_{n_1}(u_{i_1})\cdots\kq_{n_k}(u_{i_k})\vac\right)=\left(\prod_{j=1}^k\lambda_{i_j}\right)\kq_{n_1}(u_{i_1})\cdots\kq_{n_k}(u_{i_k})\vac
$$
donc $\rho(f^{[n]})=\rho(f)^n$ puis $\ent(f^{[n]})=n\cdot\ent(f)$.
\end{proof}

\section{Etude des points fixes}

\subsection{Calcul diff\'erentiel sur l'espace de Douady de points}

Soit $X$ un espace analytique, $x\in X$ et $\pzm_{X,x}\subset\cO_{X,x}$ l'id\'eal maximal des germes de fonctions holomorphes s'annulant en $x$. L'espace tangent de $X$ en $x$ est par d\'efinition :
$$
T_xX:=\Hom_\IC(\pzm_{X,x}/\pzm_{X,x}^2,\IC).
$$
Soit $\cD$ l'espace analytique r\'eduit au point muni de l'alg\`ebre des nombres duaux $\IC[\varepsilon]/\varepsilon^2$ et notons $X(\cD)_x$ l'ensemble des morphismes $t\colon\cD\to X$ envoyant le point de $\cD$ sur $x$. On a un isomorphisme d'espaces vectoriels sur $K$ :
\begin{equation}
\label{tangent1}
T_xX\cong X(\cD)_x.
\end{equation}
Soit $Y$ un autre espace analytique, $f:X\to Y$ un morphisme et $y:=f(x)$. Notons $f^\#_y\colon\cO_{Y,y}\to\cO_{X,x}$ le morphisme d'alg\`ebres locales induit. Puisque $f^\#_y(\pzm_{Y,y})\subset\pzm_{X,x}$, il induit une application lin\'eaire $\overline{f^\#_y}\colon\pzm_{Y,y}/\pzm_{Y,y}^2\to\pzm_{X,x}/\pzm_{X,x}^2$ dont le dual est par d\'efinition l'application tangente \`a $f$ en $x$ 
$$
T_xf\colon T_xX\to T_yY, \quad (T_xf)(v)=v\circ\overline{f^\#_y}.
$$
Si l'on utilise la description (\ref{tangent1}) de l'espace tangent, il est facile de voir que $T_xf$ est donn\'ee par la formule
\begin{equation}
\label{tangent2}
(T_xf)(t)=f\circ t.
\end{equation}

Si $X=S^{[n]}$ est l'espace de Douady repr\'esentant le foncteur des familles plates de sous-espaces analytiques de $S$ de dimension nulle et de longueur $n$, on sait qu'en un point $\xi\in X$ correspondant au sous-espace analytique $\xi\subset S$ de faisceau d'id\'eaux $\cI_\xi\subset\cO_S$ on a
\begin{equation}
\label{tangent3}
T_\xi X\cong\Hom_{\cO_S}\left(\cI_\xi,\cO_S/\cI_\xi\right).
\end{equation}
Si $\Sigma$ est une autre surface analytique, $Y=\Sigma^{[n]}$ et $f\colon S\to\Sigma$ est un isomorphisme, l'isomorphisme $f^{[n]}\colon X\to Y$ est donn\'e fonctoriellement ainsi : si $T$ est un espace analytique et $\cZ\subset S\times T$ une famille de sous-espaces analytiques de $S$ de dimension nulle et de longueur $n$, plate sur $T$, alors $f^{[n]}(\cZ)=(f\times \id_T)(\cZ)\in Y(T)$. La description (\ref{tangent1}) de l'espace tangent montre que, de ce point de vue, en notant $\zeta:=f(\xi)$,
\begin{equation}
\label{tangent4}
(T_\xi f^{[n]})(\cZ)=(f\times \id_\cD)(\cZ)\in T_\zeta Y,
\end{equation}
pour tout famille plate $\cZ\subset S\times\cD$ telle que $\cZ_{|S\times *}=\xi$. Avec la description (\ref{tangent3}), on calcule que $T_\xi f^{[n]}\colon\Hom_{\cO_S}\left(\cI_\xi,\cO_S/\cI_\xi\right)\to\Hom_{\cO_\Sigma}\left(\cI_\zeta,\cO_\Sigma/\cI_\zeta\right)$ est donn\'ee par la formule
\begin{equation}
\label{tangent5}
(T_\xi f^{[n]})(\phi)=(f^\#)^{-1}\circ f_*(\phi)\circ f^\#_{|\cI_\xi}.
\end{equation}
Pr\'ecis\'ement, pour $\phi\colon\cI_\xi\to\cO_S/\cI_\xi$, $(T_\xi f^{[n]})(\phi)$ est la compos\'ee
\begin{equation}
\label{tangent6}
\cI_\zeta\xrightarrow{f^\#} f_*\cI_\xi\xrightarrow{f_*(\phi)}f_*\left(\cO_S/\cI_\xi\right)\underset{(a)}{\xrightarrow{\sim}} f_*\cO_S/f_*\cI_\xi\xrightarrow{(f^\#)^{-1}}\cO_\Sigma/\cI_\zeta
\end{equation}
o\`u $(a)$ est d\^u au fait que $f$ est un morphisme affine (puisque c'est un isomorphisme), donc $R^1f_*\cI_\xi=0$. En effet, l'isomorphisme (\ref{tangent3}) se construit par recollement sur des ouverts affines, donc on peut supposer que $S$ et $\Sigma$ sont affines. Soit $A:=\cO(S)$, $B:=\cO(\Sigma$), $I:=\cI_\xi\subset A$ et $J:=\cI_\zeta\subset B$. L'id\'eal $I(\cZ)$ de $\cZ$ (dans (\ref{tangent4})) est engendr\'e par des \'el\'ements de la forme $\alpha_i+\varepsilon\beta_i$ avec $\alpha_i,\beta_i\in A$ o\`u les $\alpha_i$ engendrent l'i\'eal $I$ et l'isomorphisme (\ref{tangent3}) consiste \`a faire correspondre \`a $\cZ$ le morphisme $\phi\colon I\to A/I$ d\'efini par $\phi(\alpha_i)=\beta_i$ (ce qui est possible par platitude, voir Eisenbud{\&}Harris \cite{EisenbudHarris} ou Douady \cite{Douady}). Soit alors $\phi\in\Hom_A(I,A/I)$. La famille plate $\cZ$ correspondante a pour id\'eal $I(Z)=\langle \alpha+\varepsilon\phi(\alpha),\alpha\in I\rangle$ donc $(f\times \id_\cD)(\cZ)$ a pour id\'eal $(f^\#\times \id)^{-1}I(\cZ)$, o\`u $f^\#\colon B\xrightarrow{\sim}A$ est tel que $(f^\#)^{-1}(I)=J$. Donc 
\begin{align*}
(f^\#\times \id)^{-1}I(\cZ)&=\langle (f^\#)^{-1}(\alpha)+\varepsilon(f^\#)^{-1}(\phi(\alpha)),\alpha\in I\rangle\\
&=\langle \gamma+\varepsilon(f^\#)^{-1}(\phi(f^\#(\gamma))),\gamma\in J\rangle
\end{align*}
et \`a cette famille correspond le morphisme
\begin{equation}\label{tangent7}
(T_\xi f^{[n]})(\phi)\colon J\xrightarrow{f^\#}I\xrightarrow{\phi}A/I\xrightarrow{(f^\#)^{-1}} B/J.
\end{equation}

Notons les projections de la famille universelle ainsi :
$$
\xymatrix{\Xi_n\ar[r]^q\ar[d]^p&S\\S^{[n]}}
$$
A tout fibr\'e vectoriel $F$ sur $S$ on associe le fibr\'e vectoriel \emph{tautologique} $F^{[n]}:=p_*q^*F$ sur $S^{[n]}$ (il est localement libre car $p$ est plat) de rang $\rg(F^{[n]})=n\cdot\rg(F)$. La fibre de $F^{[n]}$ en $\xi\in S^{[n]}$ est $F^{[n]}(\xi)=\Gamma(\xi,F)$. En particulier 
$$
(T_S)^{[n]}(\xi)=\Gamma(\xi,T_S)=\Hom_{\cO_S}(\Omega_S,\cO_\xi)=\Der_\IC(\cO_S,\cO_\xi).
$$
Si $f\colon S\to\Sigma$ est un isomorphisme et $\zeta:=f(\xi)$ on a un isomorphisme
$$
(T_f)^{[n]}(\xi)\colon (T_S)^{[n]}(\xi)\xrightarrow{\sim}(T_\Sigma)^{[n]}(\zeta)
$$
d\'efini ainsi : \`a toute section $\sigma\colon\xi\to T_S$ on associe la compos\'ee
\begin{equation}
\label{tangent8}
(T_f)^{[n]}(\xi)(\sigma)\colon\zeta\xrightarrow{f^{-1}}\xi\xrightarrow{\sigma}T_S\xrightarrow{T_f}T_\Sigma.
\end{equation}
On v\'erifie ais\'ement qu'avec la description en termes de d\'erivations, l'application $(T_f)^{[n]}(\xi)\colon\Der_\IC(\cO_S,\cO_\xi)\to\Der_\IC(\cO_\Sigma,\cO_\zeta)$ est d\'efinie pour $\delta\in\Der_\IC(\cO_S,\cO_\xi)$ par la compos\'ee
\begin{equation}
\label{tangent9}
(T_f)^{[n]}(\xi)(\delta)\colon\cO_\Sigma\xrightarrow{f^\#}f_*\cO_S\xrightarrow{f_*\delta}f_*\cO_\xi\xrightarrow{(f^\#)^{-1}}\cO_\zeta.
\end{equation}
On dispose par ailleurs d'une application naturelle
$$
(T_S)^{[n]}(\xi)=\Der_\IC(\cO_S,\cO_\xi)\xrightarrow{\iota_S(\xi)}\Hom_{\cO_S}(\cI_\xi,\cO_\xi)=T_\xi S^{[n]},\quad \delta\mapsto\delta_{|\cI_\xi}
$$

\begin{remark}
Cette construction m'a \'et\'e expliqu\'ee par Manfred Lehn. A ma connaissance, elle ne figure pas dans la litt\'erature.
\end{remark}

Observons que les deux calculs de diff\'erentielle sont compatibles :

\begin{lemma}
Le diagramme suivant est commutatif :
$$
\xymatrix{(T_S)^{[n]}(\xi)\ar[r]^{\iota_S(\xi)}\ar[d]_{(T_f)^{[n]}(\xi)}& T_\xi S^{[n]}\ar[d]^{T_\xi f^{[n]}}\\
(T_\Sigma)^{[n]}(\zeta)\ar[r]^{\iota_\Sigma(\zeta)}& T_\zeta\Sigma^{[n]}}
$$
\end{lemma}

\begin{proof}
Soit $\delta\in(T_S)^{[n]}(\xi)$ une d\'erivation. Avec (\ref{tangent6}) on calcule que $(T_\xi f^{[n]}\circ\iota_S(\xi))(\delta)$ est le morphisme
$$
\cI_\zeta\xrightarrow{f^\#} f_*\cI_\xi\xrightarrow{f_*(\delta_{|\cI_\xi})}f_*\cO_\xi\xrightarrow{(f^\#)^{-1}}\cO_\zeta
$$ 
tandis qu'avec (\ref{tangent8}) $(\iota_\Sigma(\zeta)\circ(T_f)^{[n]}(\xi))(\delta)$ est le morphisme
$$
\cI_\zeta\hookrightarrow \cO_\zeta\xrightarrow{f^\#}f_*\cO_S\xrightarrow{f_*\delta}f_*\cO_\xi\xrightarrow{(f^\#)^{-1}}\cO_\zeta.
$$
Ces deux morphismes sont \'egaux puisque $f^\#(\cI_\zeta)=f_*\cI_\xi$.
\end{proof}

\begin{proposition}
Soit $\xi\in S^{[n]}$. Si $\xi\subset S$ est r\'eduit, alors $\iota_S(\xi)$ est un isomorphisme.
\end{proposition}

\begin{proof}
Puisque $S$ est une surface lisse, $S^{[n]}$ est lisse de dimension $2n$ et $\dim T_\xi S^{[n]}=2n=\dim (T_S)^{[n]}(\xi)$. Il suffit donc de montrer que $\iota_S(\xi)$ est injective. Puisque le support de $\xi$ est fini, on peut se placer sur un ouvert affine de $S$ contenant $\xi$, voire supposer que $S$ est affine. Le r\'esultat d\'ecoule alors du lemme suivant avec $A=\cO(S)$, $I=\cI_\xi$ produit d'id\'eaux maximaux deux \`a deux distincts si $\xi$ est r\'eduit :

\begin{lemma}
 Soit $K$ un corps, $A$ une $K$-alg\`ebre, $I\subset A$ un id\'eal et $\delta\colon A\to A/I$ une $K$-d\'erivation. Si $I$ est un produit d'id\'eaux maximaux deux \`a deux distincts, alors $\delta_{|I}=0$ implique $\delta=0$.
\end{lemma}

\begin{proof}[D\'emonstration du lemme]
Notons $I=\prod_{i=1}^n \pzm_i$, les $\pzm_i$ \'etant des id\'eaux maximaux deux \`a deux distincts. Si $n=0$ le r\'esultat est trivial et si $n=1$ il est clair gr\^ace \`a la d\'ecomposition vectorielle $A\cong \pzm\oplus K$. Nous supposons donc que $n\geq 2$. Soit $a\in A$, que nous d\'ecomposons de mani\`ere unique pour tout $i$ sous la forme $a=x_i+\lambda_i$ avec $x_i\in\pzm_i$, $\lambda_i\in K$. Rappelons que $A/I\cong\bigoplus_{i=1}^n A/\pzm_i$. Puisque $x_1\cdot\ldots\cdot x_n\in I$ on calcule 
$$
0=\delta(x_i\cdot\ldots\cdot x_n)=\sum_{i=1}^n x_1\cdot\ldots\cdot \widehat{x_i}\cdot\ldots\cdot x_n \delta(x_i)
$$
o\`u $x_1\cdot\ldots\cdot \widehat{x_i}\cdot\ldots\cdot x_n \delta(x_i)$ est la composante dans $A/\pzm_i$. En notant que $x_j=x_i+(\lambda_i-\lambda_j)$ on obtient
$$
x_1\cdot\ldots\cdot \widehat{x_i}\cdot\ldots\cdot x_n \delta(x_i)=\prod_{j\neq i}(\lambda_i-\lambda_j)\delta(x_i)\in A/\pzm_i.
$$
Si $\prod\limits_{j\neq i}(\lambda_i-\lambda_j)=0$ pour tout $i$, alors $\lambda_1=\ldots=\lambda_n$ donc $x_1=\ldots=x_n\in\bigcap_{i=1}^n\pzm_i=~I$, donc $\delta(a)=0$. Sinon, l'un des $x_i$ est tel que $\delta(x_i)=0$ donc $\delta(a)=0$.
\end{proof}
\end{proof}

Il r\'esulte de cette \'etude qu'en un point $\xi\in S^{[n]}$ tel que $\xi\subset S$ est r\'eduit, on peut facilement calculer la diff\'erentielle de $f^{[n]}$ en $\xi$ par la formule (\ref{tangent7}) qui ne d\'epend que de $f$. En particulier, si $f\colon S\xrightarrow{\sim}S$ est un automorphisme et $\xi\in S^{[n]}$ un point fixe de $f^{[n]}$ de support r\'eduit, alors la diff\'erentielle de $f^{[n]}$ en $\xi$ est l'application associant \`a une section $\sigma\colon\xi\to T_S$ la section $T_f\circ\sigma\circ f^{-1}\colon\xi\to T_S$.

\subsection{Etude des points fixes}

Si $X$ est une vari\'et\'e analytique complexe lisse, $f\in\Aut(X)$ et $x\in X$ un point fixe de $f$, rappelons qu'il est dit \emph{non d\'eg\'en\'er\'e} si $\det(T_xf -\id)\neq 0$. Si $f$ est lin\'earisable au voisinage de $x$ (typiquement quand $f$ est d'ordre fini) et si $x$ est non d\'eg\'en\'er\'e, alors il est isol\'e.

Soit $f\in\Aut(S)$ et $n\geq 2$. Si $F\subset S$ est une sous-vari\'et\'e fixe par $f$, la sous-vari\'et\'e $F^{[n]}\subset S^{[n]}$ n'est en g\'en\'eral pas une composante de points fixes de $f^{[n]}$ : le probl\`eme vient essentiellement des points \'epais. Pr\'ecis\'ement, les sous-espaces analytiques $\xi\in~S^{[n]}$ fixes sous $f^{[n]}$ sont r\'eunions \`a supports disjoints de sous-espaces \emph{r\'eduits} de la forme :
\begin{itemize}
\item $\{x\}$ o\`u $x$ est un point fixe de $f$,
\item $\{x,f(x),\ldots,f^{k-1}(x)\}$ o\`u $x$ est un point p\'eriodique d'ordre $k$ de $f$,
\end{itemize}
et \emph{potentiellement} de sous-espaces \emph{non r\'eduits} (o\`u \emph{\'epais}) de la forme :
\begin{itemize}
\item $\{\xi_x\}$ o\`u $x$ est un point fixe de $f$ et $\xi_x$ un sous-espace \'epais port\'e en $x$,
\item $\{\xi_x,f(\xi_x),\ldots,f^{k-1}(\xi_x)\}$ o\`u $x$ est un point p\'eriodique d'ordre $k$ de $f$ et $\xi_x$ un sous-espace \'epais port\'e en $x$,
\end{itemize}
de telle sorte que la somme des points fixes et des longueurs des orbites, compt\'es avec leur multiplicit\'e \'eventuelle, fasse $n$.
Pour calculer $T_\xi f^{[n]}$ en un tel point, il suffit de savoir le faire pour chaque type possible, en vertu du lemme suivant.

\begin{lemma}
Si $\xi_1,\ldots,\xi_k$ sont de l'un des quatre types pr\'ec\'edents et de supports disjoints, avec $\xi_i\in S^{[n_i]}$ et $\xi=\xi_1\cup\ldots\cup\xi_k\in S^{[n]}$, alors $T_\xi S^{[n]}\cong\bigoplus_{i=1}^k T_{\xi_i}S^{[n_i]}$ et pour cette d\'ecomposition $T_\xi f^{[n]}=\bigoplus_{i=1}^kT_{\xi_i}f^{[n_i]}$.
\end{lemma}

\begin{proof}
Si les $\xi_i$ sont tous r\'eduits, $\Gamma(\xi,S)=\bigoplus_{i=1}^k \Gamma(\xi_i,S)$ et $T_\xi f^{[n]}$ respecte cette d\'ecomposition donc le r\'esultat est clair. Dans le cas g\'en\'eral, il suffit d'observer que si $I$ et $J$ sont deux id\'eaux d'un anneau $A$, tels que $I+J=A$, alors on a un isomorphisme de $A$-modules :
$$
\Hom_A(IJ,A/IJ)\cong\Hom_A(I,A/I)\oplus\Hom_A(J,A/J).
$$
En effet, on a $A/IJ\cong A/I\oplus A/J$, dont nous notons $p_1$ et $p_2$ les projections. Prenons $\alpha\in I$ et $\beta\in J$ tels que $\alpha+\beta=1$. A tout $\theta\in\Hom_A(IJ,A/IJ)$ on fait correspondre $\phi(x)=p_1\theta(x\beta)\in\Hom_A(I,A/I)$ et $\psi(y)=p_2\theta(\alpha y)\in\Hom_A(J,A/J)$. R\'eciproquement, \`a un couple $(\phi,\psi)$ on associe $\theta$ d\'efinie pour tout $x\in I$ et $y\in J$ par $\theta(xy)=\phi(x)y+\psi(y)x$ vu comme \'el\'ement de $A/IJ$. On v\'erifie ais\'ement que ces applications sont inverses l'une de l'autre et que la diff\'erentielle, calcul\'ee avec la formule (\ref{tangent7}), respecte la d\'ecomposition.
\end{proof}

\subsubsection{Points fixes r\'eduits} 

\begin{proposition}
Si $x_1,\ldots,x_n$ sont des points fixes isol\'es et non d\'eg\'en\'er\'es de $f$, alors $\xi:=\{x_1,\ldots,x_n\}\in S^{[n]}$ est un point fixe isol\'e et non d\'eg\'en\'er\'e de $f^{[n]}$.
\end{proposition}

\begin{proof}
Montrons que $\xi$ est un point fixe isol\'e. En se pla\c{c}ant dans un voisinage affine de $\xi$, si $\xi$ n'est pas isol\'e il existe une suite $(\xi^i)$ de points fixes de $f^{[n]}$ convergeant vers $\xi$. Quitte \`a restreindre le voisinage, on peut supposer que tous les $\xi^i$ sont des sous-espaces r\'eduits, et on pose $\xi^i=\{x_1^i,\ldots,x_n^i\}$. Les $0$-cycles $\sum_{j}x_j^i\in S^{(n)}$ correspondant via $\rho$ convergent donc vers $\rho(\xi)$ et puisque le rev\^etement $S^n\to S^{(n)}$ est non ramifi\'e au-dessus de $S^{(n)}\setminus D$, on peut choisir des rel\`evements de cette suite telle que, dans $S^n$, la suite $(x_1^i,\ldots,x_n^i)$ converge vers $(x_1,\ldots,x_n)$. Puisque $f^{[n]}(\xi^i)=\xi^i$, pour tout $i$ il existe une transposition $\tau_i\in\kS_n$ telle que $f(x_j^i)=x^i_{\tau_i(j)}$. La suite $i\mapsto\tau_i$ est \`a valeurs dans un ensemble fini, donc elle prend au moins une valeur une infinit\'e de fois. Quitte \`a extraire une sous-suite, on peut donc supposer que $\tau_i=\tau$ est constante. Par passage \`a la limite, $f(x_j)=x_{\tau(j)}$ pour tout $j$, ce qui force $\tau=\id$ puisque les $x_j$ sont fixes par $f$. On a donc trouv\'e des suites de points fixes tendant vers les $x_i$, contradiction.

Montrons que $\xi$ est non d\'eg\'en\'er\'e. Soit $\sigma\in T_\xi S^{[n]}$. Puisque $\xi$ est r\'eduit, $\sigma\in~\Gamma(\xi,T_S)$ et si $T_\xi f^{[n]}(\sigma)=\sigma$ cela signifie que $T_f\circ \sigma\circ f^{-1}=\sigma$, mais $f_{|\xi}=\id$ donc $T_f\circ \sigma=\sigma$ ce qui donne $(T_{x_j}f)(\sigma(x_j))=\sigma(x_j)$ pour tout $j$ et contredit l'hypoth\`ese de non d\'eg\'en\'erescence des $x_j$.
\end{proof}

\begin{proposition}
\label{prop:periodique}
Si $x$ est un point p\'eriodique d'ordre $n\geq 2$ de $f$, alors son orbite $\xi:=~\{x,f(x),\ldots,f^{n-1}(x)\}\in S^{[n]}$ est un point fixe d\'eg\'en\'er\'e de $f^{[n]}$ tel que  $\dim\ker(T_\xi f^{[n]}-~\id)=~2$.
\end{proposition}

\begin{proof}
Soit $\sigma\in\Gamma(\xi,T_S)$ un vecteur tangent. Si $\sigma=T_\xi f^{[n]}(\sigma)$, alors pour tout $i=0,\ldots,n$ on a
$$
\sigma(f^i(x))=(T_f\circ\sigma\circ f^{-1})(f^i(x))=(T_{f^{i-1}(x)}f)(\sigma(f^{i-1}(x)))
$$
donc $\sigma$ est enti\`erement caract\'eris\'e par $\sigma(x)\in T_x S$, donc $\ker(T_\xi f^{[n]}-\id)\cong T_x S$.
\end{proof}

\begin{remark}
Si $f$ est d'ordre $p$ premier et si $n<p$, $f^{[n]}$ n'a pas de point fixe $\xi\in~S^{[n]}$ contenant des orbites de points p\'eriodiques d'ordre $k$ tel que $2\leq k\leq n$. Par contre, si $f$ est d'ordre fini et admet des points p\'eriodiques d'ordre $n$, la proposition~\ref{prop:periodique} montre que leur lieu dans $S^{[n]}$ est contenu dans une surface fixe.
\end{remark}

\subsubsection{Points fixes \'epais}

Soit $x\in S$ et $B_n(x):=\rho^{-1}(n\cdot x)$ la sous-vari\'et\'e des sous-espaces \'epais port\'es en $x$.  C'est une sous-vari\'et\'e irr\'eductible de dimension $n-1$  (voir Brian\c{c}on \cite{Briancon}).

Nous supposons \`a partir de maintenant que l'automorphisme $f$ de $S$ est d'ordre fini. Si $x$ est un point fixe de $f$, on peut donc lin\'eariser l'action au voisinage de $x$, ce qui ram\`ene l'\'etude des sous-espaces \'epais fixes par $f^{[n]}$ au cas o\`u $S=\IC^2$, $x=(0,0)\in\IC^2$ et $f$ est une application lin\'eaire d'ordre fini. Choisissons un syst\`eme de coordonn\'ees $X,Y$ dans lequel la matrice de $f$ est diagonale et notons $\varepsilon_1,\varepsilon_2$ ses valeurs propres (ce sont des racines de l'unit\'e). Si $\xi$ est un sous-espace port\'e en $(0,0)$, on distingue deux cas simples : 
\begin{itemize}
\item l'id\'eal $I_\xi$ est \emph{monomial},
\item l'id\'eal $I_\xi$ est \emph{curvilin\'eaire}.
\end{itemize}
Le nombre d'id\'eaux monomiaux est \'egal au nombre de partitions de l'entier $n$ tandis que les points curvilin\'eaires sont denses dans $B_n(x)$. Si $n\leq 3$, tout point \'epais est de l'un de ces types.

\begin{proposition} 
Si $f$ est d'ordre fini, tout sous-espace \'epais d'id\'eal monomial est fixe sous $f^{[n]}$.
\end{proposition}

\begin{proof}
Soit $\lambda=(\lambda_1\geq\cdots\geq\lambda_k)$ une partition de $n$. Notons son diagramme de Young par $D(\lambda):=\{(i,j)\in\IN\times\IN\,|\, i<\lambda_{j+1}\}$.  L'id\'eal monomial de partition $\lambda$ s'\'ecrit $I_\lambda=\langle X^iY^j\,|\,(i,j)\notin D(\lambda)\rangle$. Puisque $f$ agit diagonalement dans ce syst\`eme de coordonn\'ees, l'id\'eal $I_\lambda$ est clairement invariant par $f$.
\end{proof}

En conservant les notations de la d\'emonstration, une base du quotient $\IC[X,Y]/I_\lambda$ est constitu\'ee des mon\^omes $X^uY^v$ tels que $(u,v)\in D(\lambda)$. Un syst\`eme minimal de g\'en\'erateurs de $I_\lambda$ parmi les $X^iY^j$ tels que $(i,j)\notin D(\lambda)$ se trouve en ne conservant que les indices $(i,j)$ figurant \`a une \emph{marche} du bord du domaine $(\IN\times\IN)\setminus D(\lambda)$. Notons $G(\lambda)$ cet ensemble d'indices. Un vecteur tangent au point d\'efini par cet id\'eal correspond \`a un morphisme $\varphi\in\Hom_{\IC[X,Y]}(I_\lambda,\IC[X,Y]/I_\lambda)$ et est donc donn\'e par une matrice $(\alpha^{u,v}_{i,j})$ d\'efinie par $\varphi(X^iY^j)=\sum_{u,v}{\alpha^{u,v}_{i,j}X^uY^v}$ pour $(i,j)\in D(\lambda)$ et $(u,v)\in G(\lambda)$ (les coefficients $\alpha^{u,v}_{i,j}$ sont soumis \`a des relations venant des relations entre les g\'en\'erateurs de $I_\lambda$). La diff\'erentielle de $f^{[n]}$ en $\varphi$, calcul\'ee par la formule (\ref{tangent7}), envoie la matrice $(\alpha^{u,v}_{i,j})$ sur $(\varepsilon_1^{i-u}\varepsilon_2^{j-v}\alpha^{u,v}_{i,j})$. S'il n'existe pas d'indices $(i,j)\in G(\lambda)$ et $(u,v)\in D(\lambda)$ tels que $\varepsilon_1^{i-u}\varepsilon_2^{j-v}=1$, alors le point fixe est non d\'eg\'en\'er\'e et isol\'e.

\bigskip

Etudions maintenant les points \'epais curvilin\'eaires. Pour $n=2$, ils sont de la forme $I_{(\lambda:\mu)}:=\langle \lambda x+\mu y,x^2,y^2\rangle$ pour $(\lambda:\mu)\in\IP^1_\IC\setminus\{(1:0),(0:1)\}$ (les points exclus donnent les deux id\'eaux monomiaux de la fibre exceptionnelle). Alors $(f^\#)^{-1}(I_{(\lambda:\mu)})=I_{(\varepsilon_1^{-1}\lambda:\varepsilon_2^{-1}\mu)}$ est fixe si et seulement si $\varepsilon_1=\varepsilon_2$, auquel cas tous les id\'eaux curvilin\'eaires de longueur deux sont fixes. Pour $n\geq 3$, ils sont de la forme $I_{\underline{\alpha}}:=\langle y+\alpha_1 x+\cdots+\alpha_{n-1}x^{n-1},x^n\rangle$ (ou avec $x$ et $y$ \'echang\'es), avec $\underline{\alpha}=(\alpha_1,\ldots,\alpha_{n-1})\in\IC^{n-1}$ non nul. On calcule de m\^eme que $(f^\#)^{-1}(I_{\underline{\alpha}})=I_{\underline{\alpha}'}$ avec $\alpha'_i=\alpha_i\varepsilon_2\varepsilon_1^{-i}$. Le point est donc fixe si et seulement si pour chaque indice $i$ tel que $\alpha_i$ est non nul, on a $\varepsilon_2=\varepsilon_1^i$. Ainsi, s'il existe un indice $i\in\{1,\ldots,n-1\}$ tel que $\varepsilon_2=\varepsilon_1^i$, alors tous les points curvilin\'eaires de la forme $\underline{\alpha}=(0,\ldots,0,\alpha_i,0\ldots,0)$ sont fixes. En particulier, un point curvilin\'eaire fixe n'est jamais isol\'e. 

L'\'etude des points fixes constitu\'es de l'orbite d'un point \'epais et plus d\'elicate. Si $f$ est d'ordre premier $p$, pour tout $x\in S$ et tout point \'epais $\xi$ de longueur $\ell$ et de support $x$, l'orbite $\{\xi,f(\xi),\ldots,f^{p-1}(\xi)\}\in S^{[p\ell]}$ est un point fixe non isol\'e : ce lieu de points fixes est de dimension $\ell-1$ si $x$ est fix\'e et $\ell+1$ si $x$ varie dans $S$.

\begin{example}
Si $S$ est un tore complexe et $f$ l'involution $f(x)=-x$, elle admet exactement $16$ points fixes isol\'es dont l'action locale en chacun d'eux a pour valeurs propres $\varepsilon_1=\varepsilon_2=-1$. Donc tous les points \'epais de longueur deux sont fixes sous $f^{[2]}$ dans $S^{[2]}$. Consid\'erons la surface de $S^{[2]}$ constitu\'ee des orbites $\{x,f(x)\}$ lorsque $x$ n'est pas un point fixe de $f$. Sa fermeture contient donc les $16$ courbes de points \'epais port\'es en les points fixes de $f$, donc s'identifie \`a la surface de Kummer $K^2(S)$, d\'efinie comme la fibre au-dessus de $0$ de la compos\'ee $S^{[2]}\xrightarrow{\rho}S^{(2)}\xrightarrow{+}S$, qui est la r\'esolution minimale du quotient $S/f$. Notons que les points fixes de $f^{[2]}$ de la forme $\{x,x'\}$ o\`u $x$ et $x'$ sont deux points fixes distincts de $f$ restent en-dehors et donnent $120$ points fixes isol\'es.
\end{example}

\begin{example}
Soit $S$ une surface K3 alg\'ebrique admettant un automorphisme symplectique (\ie laissant invariante la forme symplectique) $f$ d'ordre premier $p$ valant $3$, $5$ ou $7$. D'apr\`es Nikulin \cite{Nikulin}, $f$ admet un nombre fini $m_p$ de points fixes isol\'es valant respectivement $6$, $4$ et $3$. Consid\'erons l'action de $f^*$ sur $H^2(S,\IC)$ : notons $a_p$ la multiplicit\'e de la valeur propre $1$ et $b_p$ celle de $\xi^i$, $i=1,\ldots,p-1$, o\`u $\xi$ est une racine primitive $p$-i\`eme de l'unit\'e (elles ont toutes la m\^eme multiplicit\'e). D'apr\`es Garbagnati\&Sarti \cite[Proposition 1.1]{GaSa}, leurs valeurs sont : $a_3=10, b_3=6$, $a_5=6,b_5=4$, $a_7=4,b_7=3$. L'action de $f^{[2]}$ sur $S^{[2]}$ n'a que des points fixes isol\'es et non d\'eg\'en\'er\'es (l'action locale de $f$ a pour valeurs propres $\varepsilon_1=\xi$ et $\varepsilon_2=\bar{\xi}$), leur nombre est $\frac{m_p(m_p-1)}{2}+2m_p$, respectivement $27$, $14$ et $9$ (paires de points fixes et points \'epais monomiaux de multiplicit\'e deux), nombre que l'on obtient encore avec la formule de Lefschetz en utilisant la proposition \ref{prop:traceFock}. Par contre, pour $n=3$ et $p=5$, la m\^eme formule donne $36$ points fixes isol\'es : on en compte en effet $4$ form\'es de triplets de points fixes, $24$ form\'es d'un point fixe r\'eduit et d'un point fixe double, mais parmi les $12$ points triples monomiaux, $4$ sont d\'eg\'en\'er\'es (ceux associ\'es \`a la partition $(1,1,1)$) et il y a des composantes de points curvilin\'eaires fixes.
\end{example}

\bibliographystyle{smfplain}
\bibliography{BiblioAutHilb}

\end{document}